\documentclass[11pt]{article}

\usepackage[
  backend=biber,
  style=alphabetic,
  sorting=nyt,
  doi=false,
  url=false,
  eprint=true,
  maxbibnames=99
]{biblatex}

\usepackage[margin=1in]{geometry}
\usepackage{amsmath,amssymb,amsthm,mathtools}
\usepackage{xcolor}
\usepackage[colorlinks=true,linkcolor=blue,citecolor=blue,urlcolor=blue]{hyperref}
\usepackage{microtype}

\definecolor{markcol}{HTML}{087E8B}
\definecolor{edgecol}{HTML}{4B5563}
\definecolor{branchfill}{HTML}{CDE9EA}

\numberwithin{equation}{section}

\theoremstyle{plain}
\newtheorem{theorem}{Theorem}
\newtheorem{lemma}[theorem]{Lemma}

\theoremstyle{definition}
\newtheorem{definition}[theorem]{Definition}

\title{Square-Difference-Free Sets beyond the Three-Quarter Barrier}
\author{Dmitry Krachun%
  \thanks{Princeton University. Email:
  \href{mailto:dk9781@princeton.edu}{\texttt{dk9781@princeton.edu}}}}
\date{}

\begin{document}

\maketitle

\begin{abstract}
Let $D(N)$ denote the largest cardinality of a subset of $\{1,\ldots,N\}$ containing no nonzero square difference. While a construction certifying $D(N)\geq (1-o(1))N^{1/2}$ is almost trivial, Erd\H{o}s conjectured that this bound is sharp up to polylogarithmic factors. This was disproved by S\'ark\"ozy and later again by Ruzsa, who found an elegant construction showing that $D(N)\geq c\cdot N^{0.733077\dots}$, with an absolute constant $c>0$. His approach was subsequently refined, leading to the previously best known lower bound with exponent $0.7334117\dots$ due to Beigel–Gasarch and, independently, Lewko. However, in the original paper Ruzsa observed that $3/4$ seems to be the natural barrier of his approach. 

In this paper we develop a new construction leading to the lower bound 
\[
 \liminf_{N\to\infty}\frac{\log D(N)}{\log N}
 \geq \alpha_*:= 0.7527964558\ldots;
\]
thus crossing the natural exponent-$3/4$ barrier of Ruzsa's method. The value $0.7527964558\ldots$ arises from a simple optimisation problem and appears to be the limit of the new approach.
\end{abstract}

\section{Introduction}

A set of natural numbers is \emph{square-difference-free} if the absolute difference between any two distinct elements is not a square. Write $D(N)$ for the maximal size of a square-difference-free subset of $\{1, 2, \dots, N\}$. Furstenberg and S\'ark\"ozy independently proved that $D(N)=o(N)$~\cite{Furstenberg1977,Sarkozy1978}. This bound was subsequently improved several times and the best currently known upper bound is $D(N)\ll N\exp(-c\sqrt{\log N})$ for some absolute $c>0$, due to Green and Sawhney~\cite[Theorem~1.1]{GreenSawhney2024}. We refer to the introduction of~\cite{GreenSawhney2024} for the history of the upper bounds. 

On the other hand, Erd\H{o}s had conjectured a much stronger bound $D(N)=O(N^{1/2}(\log N)^C)$ for some constant $C$. S\'ark\"ozy disproved this~\cite{Sarkozy1978II} and proposed instead that $D(N)=O_\varepsilon(N^{1/2+\varepsilon})$ for every $\varepsilon>0$; Ruzsa's subsequent construction disproved that prediction as well~\cite{Ruzsa1984} by showing that $D(N)>c\cdot N^{\alpha}$, where $\alpha=0.733077\dots$ and $c>0$ is an absolute constant.

% \begin{conjecture}\label{conj:gap}
% There exists $\varepsilon>0$, such that 
% \[
%  D(N)< N^{1-\varepsilon}
% \]
% holds for all $N>1$.
% \end{conjecture}

Ruzsa chooses a squarefree modulus $m$ and a set $R\subseteq\mathbb Z/m\mathbb Z$ containing no nonzero square difference. Restricting alternating digits in a base-$m$ expansion to $R$ and leaving the remaining digits free gives the exponent
\[
 \frac{1+\log_m|R|}{2};
\]
his choice $m=65$ and $|R|=7$ gives $0.733077\ldots$~\cite{Ruzsa1984}. Beigel--Gasarch and, independently, Lewko later used $m=205$ and $|R|=12$, raising this to
\[
 \frac12+\frac{\log 12}{2\log 205}
 =0.7334117970\ldots,
\]
see~\cite{BeigelGasarch2008,Lewko2015}. In his original paper~\cite{Ruzsa1984} Ruzsa also proved that $|R|<\sqrt{m}$ for all squarefree integers $m$ with all prime divisors congruent to 1 modulo 4 and further conjectured that $|R|\leq \sqrt{m}$ holds for all $m$. This makes $3/4$ a natural barrier for Ruzsa's construction. Moreover, there is some evidence that this latter exponent arising from modulus $205$ is the limit of this approach: in~\cite{Lewko2015} a maximal-clique search over all squarefree moduli $m\leq733$ was conducted and the choice $m=205$ and $|R|=12$ remained optimal in this range; more recently,  Georgiev, G\'omez-Serrano, Tao, and Wagner explicitly tasked AlphaEvolve with improving it; the system quickly recovered the same modulus $205$ but found no better example~\cite{GeorgievEtAl2025}. 

We develop a new approach also based on the base-$p$ expansion of numbers, but instead of using sets containing no nonzero square differences modulo $p$, we use a more involved construction based on sequences of residues $(s_0, \dots, s_{t-1})$ in $\mathbb{F}_p$ (with $p=4k+3$) with the property that $s_b-s_a$ is a nonzero quadratic residue modulo $p$ whenever $b>a$. We call such sequences \emph{Paley chains} as they naturally arise in Paley tournaments; see, e.g.,~\cite{Satake2021}. While for an individual prime $p$ a Paley chain never leads to an improved exponent, we efficiently combine different primes to eventually get an exponent larger than $3/4$. Specifically, our main result is the following.

\begin{theorem}\label{thm:main}
For the ten pairs
\[
 (p_i,t_i)=(3,2),(7,3),(11,4),(19,5),(23,5),
 (31,7),(43,7),(59,9),(71,9),(103,11),
\]
put $\alpha_i:=\tfrac{\log t_i}{\log p_i}$. Then 
\[
 \liminf_{N\to\infty}\frac{\log D(N)}{\log N}\geq \alpha_*,
\]
where 
\begin{equation}\label{eq:alpha-star-intro}
 \alpha_* := \frac{10+(1/\alpha_1+\cdots+1/\alpha_{10})}{1+2(1/\alpha_1+\cdots+1/\alpha_{10})}=0.752796455874514\ldots.
\end{equation}
\end{theorem}

The rest of the article is organised as follows. In Section~\ref{sec:construction} we develop the construction of a square-difference-free set based on a Paley chain in $\mathbb{F}_p$, then explain how to glue several moduli together to get an improved exponent, and finally choose ten Paley chains with small prime moduli to prove Theorem~\ref{thm:main}. In Section~\ref{sec:optimality-discussion} we then briefly discuss why the exponent $\alpha_*$ is likely to be the limit of the present method.

\section{The construction}\label{sec:construction}
\begin{definition}
For a prime $p\equiv3\pmod4$, let $(\mathbb{F}_p^\times)^2$ be the set of nonzero quadratic residues in $\mathbb F_p$. We call an ordered tuple $S=(s_0,s_1,\ldots,s_{t-1})\in\mathbb{F}_p^t$ satisfying 

\begin{equation}\label{eq:paley-chain}
 s_b-s_a\in(\mathbb{F}_p^\times)^2
 \qquad(0\leq a<b<t)
\end{equation}  
a \emph{Paley chain} in $\mathbb{F}_p$.
\end{definition}
% \begin{remark}
% Equivalently, $S$ is a transitive subtournament of the Paley tournament, with its transitive ordering displayed.
% \end{remark}

\begin{lemma}[A local ranked block]\label{lem:local-block}
Let $p\equiv3\pmod4$ be prime, let $S=(s_0,\ldots,s_{t-1})$ be a Paley chain, and let $e\geq1$ be an integer. There is a set
\[
 \mathcal{C}(p,S,e)\subseteq\mathbb Z/p^{2e}\mathbb Z,
 \qquad |\mathcal{C}(p,S,e)|=(pt)^e,
\]
and a function
\[
 h:\mathcal{C}(p,S,e)\longrightarrow\{0,1,\ldots,t^e-1\}
\]
such that, whenever $x\neq y$ in $\mathcal{C}(p,S,e)$ and $y-x$ is a square modulo $p^{2e}$, one has $h(x)>h(y)$.
\end{lemma}
\begin{proof}
Represent residues modulo $p^{2e}$ by their base-$p$ expansions
\[
 x=x_0+x_1p+\cdots+x_{2e-1}p^{2e-1},
 \qquad 0\leq x_j<p.
\]
Let $\mathcal{C}(p,S,e)$ consist of those residues for which every even-position digit $x_{2j}$ lies in $S$; all odd-position digits are unrestricted. This gives $|\mathcal{C}(p,S,e)|=(pt)^e$.

If $x_{2j}=s_{i_j}$, define
\begin{equation}\label{eq:local-height}
 h(x)=\sum_{j=0}^{e-1}(t-1-i_j)t^{e-1-j}.
\end{equation}
Note that $h(x)$ is strictly decreasing with respect to the usual lexicographic order on $(i_0,\ldots,i_{e-1})$. Suppose that $y-x$ is a nonzero square modulo $p^{2e}$, and let $r$ be the least base-$p$ position at which $x$ and $y$ differ. The $p$-adic valuation of a nonzero square modulo $p^{2e}$ is even, so $r=2j$. After division by $p^{2j}$, the leading digit $y_{2j}-x_{2j}$ is a nonzero quadratic residue modulo $p$. Write $x_{2j}=s_a$ and $y_{2j}=s_b$. Condition~\eqref{eq:paley-chain}, together with the fact that $-1$ is a nonresidue modulo $p$, forces $b>a$. Thus the first differing base-$t$ digit in \eqref{eq:local-height} is smaller for $y$ than for $x$, which implies that $h(x)>h(y)$.
\end{proof}

This lemma, while being very similar to the original approach of Ruzsa, does not directly lead to a large square-difference-free set in $\mathbb{Z}/p^{2e}\mathbb{Z}$. However, we can construct a large square-difference-free set inside $\{1, 2, \dots, p^{2e}\cdot t^{e}\}$ which projects to a translate of $\mathcal{C}(p,S,e)$ modulo $p^{2e}$. We present the construction in a slightly larger generality so that we can later apply it to several distinct primes together. 

\begin{lemma}[From a ranked block to integers]\label{lem:ranked-lift}
Let $P=q^2$ be a perfect square, and suppose that $\mathcal{C}\subseteq\mathbb Z/P\mathbb Z$ admits a function $h:\mathcal{C}\longrightarrow\{0,1,\ldots,H-1\}$ such that
\begin{equation}\label{eq:rank-descent-hypothesis}
 x, y \in \mathcal{C}, \quad x\neq y,\quad y-x\text{ a square modulo }P
 \quad\Longrightarrow\quad h(x)>h(y).
\end{equation}
Then, for every $L\geq1$, there is a square-difference-free set
\[
 \mathcal{A}_L\subseteq\{1,2,\ldots,(PH)^L\},
 \qquad |A_L|=|\mathcal{C}|^L.
\]
\end{lemma}

\begin{proof}
For a word $(x_0,\ldots,x_{L-1})\in \mathcal{C}^L$, choose representatives $0\leq\overline{x_j}<P$ and set
\[
 X=\sum_{j=0}^{L-1}\overline{x_j}P^j,
 \qquad
 h_L(X)=\sum_{j=0}^{L-1}h(x_j)H^{L-1-j}.
\]
Suppose that $Y-X$ is a nonzero square modulo $P^L$, and let $j$ be the least base-$P$ position at which $X$ and $Y$ differ. Write
\[
 Y-X=P^j(\delta+PZ),\qquad \delta\not\equiv0\pmod P.
\]
If $Y-X\equiv z^2\pmod{P^L}$, then $P^j=q^{2j}$ divides both $Y-X$ and $P^L$, so $q^{2j}\mid z^2$, and hence $q^j\mid z$. Dividing by $P^j=q^{2j}$ shows that $\delta$ is a square modulo $P$. Therefore, $y_j-x_j$ is a square modulo $P$, which implies $h(x_j)>h(y_j)$. Since $j$ is the smallest index with $x_j\neq y_j$, this, in turn, implies that $h_L(X)>h_L(Y)$.

Now put
\[
 \mathcal{B}_L=\{X+P^Lh_L(X):(x_0,\ldots,x_{L-1})\in \mathcal{C}^L\}.
\]
The reductions of two distinct elements of $\mathcal{B}_L$ modulo $P^L$ are distinct: equality of the reductions would give the same word, hence also the same rank and the same element. Thus the preceding paragraph applies to any proposed nonzero square difference. If
\[
 [Y+P^Lh_L(Y)]-[X+P^Lh_L(X)]
\]
were a positive square, reduction modulo $P^L$ would imply $h_L(X)>h_L(Y)$, and the displayed integer would be at most
\[
 (P^L-1)-P^L=-1,
\]
a contradiction. Finally, $\mathcal{B}_L\subseteq\{0,\ldots,(PH)^L-1\}$ and has $|\mathcal{C}|^L$ elements. Translating by one proves the claim.
\end{proof}

Applied directly to Lemma~\ref{lem:local-block}, Lemma~\ref{lem:ranked-lift} gives the exponent
\begin{equation}\label{eq:one-prime-exponent}
 \frac{\log(pt)}{\log(p^2t)}
 =\frac{1+\alpha}{2+\alpha},
 \qquad \alpha=\frac{\log t}{\log p}.
\end{equation}
This is independent of $e$ or $L$. The chain $S=(0,1)$ at $p=3$ gives the best possible\footnote{This follows from the fact that any Paley chain in $\mathbb{F}_p$ has length bounded by $1+\sqrt{2p-1}$~\cite{Satake2021} together with an explicit computation for $p=7, 11, 19, 23$ and $31$. See also Section~\ref{sec:optimality-discussion} for the discussion about stronger bounds.} value of $\alpha$ leading to the exponent
\[
 \frac{\log6}{\log18}=0.619906\ldots,
\]
which is well below the previously known exponent. The improvement comes from combining several prime moduli. We glue together different moduli using the Chinese remainder theorem (CRT); the moduli and block cardinalities multiply, whereas the maximum rank adds, thus improving the bound.

% Trivially $t\leq p$, so \eqref{eq:one-prime-exponent} is at most $2/3$; the Paley bounds discussed in Remark~\ref{rem:global-barrier} give $t\leq p^{1/2+o(1)}$, and hence an asymptotic ceiling of $3/5$ for this one-prime construction. The improvement beyond $3/4$ is therefore genuinely a many-modulus effect.

\begin{lemma}[Gluing ranked blocks] \label{lem:gluing}
For $1\leq i\leq\ell$, let $P_i$ be pairwise coprime perfect squares, let $\mathcal{C}_i\subseteq\mathbb Z/P_i\mathbb Z$, and suppose that
\[
 h_i:\mathcal{C}_i\longrightarrow\{0,1,\ldots,H_i-1\}
\]
strictly decreases along every nonzero square difference, i.e. if $x, y \in \mathcal{C}_i$ are distinct and $y-x$ is a square modulo $P_i$, then one has $h_i(x)>h_i(y)$. Under the Chinese remainder identification, set
\[
 P=\prod_{i=1}^{\ell}P_i,\qquad
 \mathcal{C}=\prod_{i=1}^{\ell}\mathcal{C}_i \subseteq \mathbb{Z}/P\mathbb{Z},\qquad
 h(x_1,\ldots,x_\ell)=\sum_{i=1}^{\ell}h_i(x_i),
\]
and put
\[
 H=1+\sum_{i=1}^{\ell}(H_i-1).
\]
Then $h:\mathcal{C}\longrightarrow\{0,\ldots,H-1\}$ strictly decreases along every nonzero square difference modulo $P$.
\end{lemma}

\begin{proof}
If $y-x$ is a square modulo $P$, then its reduction in every CRT coordinate is a square modulo $P_i$. A zero coordinate leaves the corresponding rank unchanged, whereas a nonzero coordinate strictly decreases it. Since $x\neq y$, at least one coordinate is nonzero. Summing the local inequalities proves $h(x)>h(y)$.
\end{proof}

Let $p_1, \dots, p_\ell$ be distinct primes congruent to 3 modulo 4. For each prime $p_i$ let $t_i$ be the length of a chosen Paley chain in $\mathbb{F}_{p_i}$. Combining Lemmas~\ref{lem:local-block} and \ref{lem:gluing} at primes $p_1, \dots, p_\ell$ gives
\begin{equation}\label{eq:block-data}
 P=\prod_{i=1}^\ell p_i^{2e_i},\qquad
 |\mathcal{C}|=\prod_{i=1}^\ell (p_it_i)^{e_i},\qquad
 H=1+\sum_{i=1}^\ell (t_i^{e_i}-1). 
\end{equation}
The integer sets supplied by Lemma~\ref{lem:ranked-lift} consequently have exponent
\begin{equation}\label{eq:block-exponent}
 \alpha(\boldsymbol e)=
 \frac{\displaystyle\sum_{i=1}^\ell e_i\log(p_it_i)}
 {\displaystyle2\sum_{i=1}^\ell e_i\log p_i+
  \log\!\left(1+\sum_{i=1}^\ell(t_i^{e_i}-1)\right)}.
\end{equation}

\begin{proof}[Proof of Theorem~\ref{thm:main}]
We use the following ten Paley chains. Every forward difference in a displayed row is a nonzero quadratic residue modulo the prime in that row.

\begin{center}
\small
\renewcommand{\arraystretch}{1.16}
\begin{tabular}{c@{\qquad}c@{\qquad}l}
$p$ & $t$ & ordered chain $S$\\
\hline
$3$   & $2$  & $(0,1)$\\
$7$   & $3$  & $(0,4,1)$\\
$11$  & $4$  & $(0,3,1,4)$\\
$19$  & $5$  & $(0,5,11,9,16)$\\
$23$  & $5$  & $(0,18,1,3,4)$\\
$31$  & $7$  & $(0,25,14,1,19,8,2)$\\
$43$  & $7$  & $(0,31,9,23,4,40,1)$\\
$59$  & $9$  & $(0,49,15,7,16,19,35,36,5)$\\
$71$  & $9$  & $(0,8,12,18,48,27,1,37,20)$\\
$103$ & $11$ & $(0,79,25,58,55,81,4,1,34,83,59)$
\end{tabular}
\end{center}

For a real parameter $U>0$ large enough, take
\begin{equation}\label{eq:asymptotic-multiplicities}
 e_i(U)=\left\lfloor\frac{U}{\log t_i}\right\rfloor.
\end{equation}
Using Lemma~\ref{lem:local-block} we can construct ten local ranked blocks with these multiplicities. Lemma~\ref{lem:gluing} combines them into a ranked set $\mathcal{C}$ modulo the perfect square $P$, with rank bounded by $H$ as in \eqref{eq:block-data}. We have $e_i(U)\log t_i=U+O(1)$ and
\[
 \log\!\left(1+\sum_i(t_i^{e_i(U)}-1)\right)=U+O(1).
\]
Substitution into \eqref{eq:block-exponent} shows that
\[
 \alpha(\boldsymbol e(U))\longrightarrow
 \frac{\displaystyle\sum_i\frac{\log(p_it_i)}{\log t_i}}
 {\displaystyle1+2\sum_i\frac{\log p_i}{\log t_i}}
 =\alpha_*.
\]
Given $\rho<\alpha_*$, choose $U$ so that $\alpha(\boldsymbol{e}(U))>\rho$, and then let $L$ tend to infinity in Lemma~\ref{lem:ranked-lift} to construct, for 
\begin{equation}\label{eq:form-of-N}
N_L=(P(U)H(U))^L=\left[\left(\prod_{i=1}^{10} p_i^{2e_i(U)}\right)\left(1+\sum_{i=1}^{10}(t_i^{e_i(U)}-1)\right)\right]^L,
\end{equation}
a square-difference-free subset $\mathcal{A}_L\subseteq \{1, 2, \dots, N_L\}$ of size at least $N_L^\rho$. Rounding arbitrary $N$ down to the largest $N_L$ not exceeding $N$, i.e. choosing $L:=\lfloor\log{N}/\log{(P(U)H(U))}\rfloor$, using the corresponding set $\mathcal{C}_L$, and finally letting $\rho \nearrow \alpha_*$ we conclude that 
\[
\liminf_{N\to\infty}\frac{\log D(N)}{\log N}\geq \alpha_*. 
\]
\end{proof}
\section{Discussion of optimality}\label{sec:optimality-discussion}

In this section we briefly discuss, without giving precise proofs, why $\alpha_*$ is likely to be the limit of the current approach based on gluing various prime moduli and Paley chains. It is not difficult to see that the choice \eqref{eq:asymptotic-multiplicities} is asymptotically optimal as $U\rightarrow\infty$ for the displayed set of ten pairs $(p_i, t_i)$. More generally, since 
\[
\log{\left(1+\sum_i (t_i^{e_i}-1)\right)}
\geq 
\max_i e_i\log{t}_i,
\]
we always have 
\[
\alpha(\boldsymbol e)=
 \frac{\displaystyle\sum_{i=1}^\ell e_i\log(p_it_i)}
 {\displaystyle2\sum_{i=1}^\ell e_i\log p_i+
  \log\!\left(1+\sum_{i=1}^\ell(t_i^{e_i}-1)\right)}
  \leq
  \frac{\displaystyle\sum_{i=1}^\ell e_i\log(p_it_i)}
 {\displaystyle2\sum_{i=1}^\ell e_i\log p_i+\max_i e_i\log{t}_i}=:\overline{\alpha}(\boldsymbol{e});
\]
and on the other hand, given any tuple $\boldsymbol{e}$ of non-negative \emph{real} numbers, we have $\alpha(\lfloor U \boldsymbol{e}\rfloor) \rightarrow \overline{\alpha}(\boldsymbol{e})$, as $U$ tends to infinity, where $\lfloor U \boldsymbol{e}\rfloor$ is obtained by scaling all coordinates of $\boldsymbol{e}$ by $U$ and then taking coordinatewise floors. Thus, the problem is reduced to finding the supremum of $\overline{\alpha}(\boldsymbol{e})$ on $\mathbb{R}_{\geq 0}^{\ell}\setminus \{0\}^\ell$. Note that even though Lemma~\ref{lem:local-block} requires $e_i\geq 1$, we can allow some coordinates of $\boldsymbol{e}$ to be equal to zero by dropping the corresponding primes.

This optimisation problem can be solved explicitly and leads to the following result. Given a sequence of pairs $(p_i, t_i)$ ordered by decreasing $\alpha_i=\tfrac{\log t_i}{\log p_i}$, an optimal choice of the exponents is given by $e_i \propto 1/\log{t_i}$ for $i\leq k$ and $e_i=0$ for $i>k$ where $k$ is chosen to maximise 
\begin{equation}\label{eq:prefix-values}
\frac{ k+(1/\alpha_1+\cdots+ 1/\alpha_k)}{1+2(1/\alpha_1+\cdots+ 1/\alpha_k)}.    
\end{equation}

Formula \eqref{eq:prefix-values} gives a simple procedure for choosing pairs $(p_i, t_i)$ to optimise the resulting exponent. For a prime $p\equiv3\pmod4$, let $t(p)$ be the size of the largest Paley chain modulo $p$. It is relatively easy to compute the values $t(p)$ for all primes $p$ not exceeding a thousand; see also~\cite{SanchezFlores1998} for the list of values. Ordering pairs $(p, t(p))$ for $3\leq p\leq 991$ by the value $\alpha(p):=\tfrac{\log t(p)}{\log p}$, we obtain
\[
\begin{aligned}
(3,2)\succ (11,4)\succ (31,7)\succ (7,3)
\succ (19,5)
\succ (59,9)\succ (103,11)\succ (43,7) 
\succ (71,9)\succ (23,5) 
\succ\cdots, %\succ (47, 7) \succ (79,9)\succ \cdots .
\end{aligned}
\]
and a direct computation shows that choosing $k=10$ maximises the value of \eqref{eq:prefix-values}, thus giving the optimal construction among primes below 1000. To turn this into a complete proof of optimality, it would be sufficient to prove that, for any $p>1000$, one has 
\begin{equation}\label{eq:desired-tp-bound}
t(p)<p^{2\alpha_*-1}=p^{0.50559291\dots}.
\end{equation}
Indeed, assuming there exists a set of primes $p_1, \dots, p_\ell$ leading to an exponent larger than $\alpha_*$, consider such a set of minimal cardinality; then the maximum is given by 
\[
\frac{\ell+(1/\alpha_1+\cdots+1/\alpha_\ell)}{1+2(1/\alpha_1+\cdots+1/\alpha_\ell)} > \alpha_*.
\]
Since we already know the optimality of $\alpha_*$ for primes below 1000, at least one of the $p_i$ must exceed 1000. Then, assuming \eqref{eq:desired-tp-bound} holds, we have $\alpha_i< 2\alpha_*-1$ for some $i\in \{1, 2, \dots, \ell\}$. Removing $p_i$ from the set will strictly increase the value of the exponent, as the numerator decreases by $1+1/\alpha_i$ and the denominator decreases by $2/\alpha_i$ and $\tfrac{1+1/\alpha_i}{2/\alpha_i}<\alpha_*$. This contradicts the minimality of the example, showing that $\alpha_*$ is the best exponent among all subsets of the primes.   

The table of values $t(p)$ for $p$ up to a thousand shows that \eqref{eq:desired-tp-bound} holds for all primes in the range $103< p < 1000$,  and, indeed, for all primes up to a thousand except the ten primes we chose. Moreover, although $\alpha(p)=\tfrac{\log{t(p)}}{\log{p}}$ is not monotone, its computed values exhibit a clear overall downward trend, dropping to $0.42$ for $p$ around 1000. This suggests that $t(p)<\sqrt p$ for every $p>103$, which would imply \eqref{eq:desired-tp-bound} for all $p>1000$. However, the best general bound available is due to Satake, who proved that $t(p)\leq 1+\sqrt{2p-1}$~\cite{Satake2021}. This bound is asymptotically better than \eqref{eq:desired-tp-bound} but worse for all $p<8.16\times 10^{26}$, leaving a huge gap.

\printbibliography

% \bibliographystyle{plain}
% \bibliography{references.bib}

\end{document}